%% file: recognition4.tex
\documentclass[12pt]{amsart}
\usepackage{geometry} 
\usepackage{amsmath}
\usepackage{amsfonts,amssymb,amsthm,amscd,latexsym,euscript}
\usepackage{mathrsfs}

\usepackage[all]{xy}
\SelectTips{cm}{10}

\input{definitions.tex}

\title{Homotopical recognition of diagram categories}
\author{Boris Chorny}
\author{David White}

\begin{document}

\begin{abstract}
Building on work of Marta Bunge in the one-categorical case, we characterize when a given model category is Quillen equivalent to a presheaf category with the projective model structure. This involves introducing a notion of \emph{homotopy atoms}, generalizing the orbits of Dwyer and Kan, \cite{DK}.  Apart from the orbit model structures of Dwyer and Kan, our examples include the classification of stable model categories after Schwede and Shipley, \cite{Schwede-Shipley}, isovariant homotopy theory after Yeakel, \cite{isovariant}, and Cat-enriched homotopy theory after Gu, \cite{gu2016generalized}.

As an application, we give a classification of polynomial functors (in the sense of Goodwillie calculus, \cite{Goo:calc3}) from finite pointed simplicial sets to spectra, and compare it to the previous work by Arone and Ching, \cite{Arone-Ching-crosseffects}.
\end{abstract}

\maketitle

\section{Introduction}

Sixty years ago, Marta Bunge, \cite{Bunge}, gave a criterion for when a category is equivalent to a functor category indexed by a small category. This is happens if and only if the category is equipped with a set of atoms. In this paper, we provide a homotopical version of Bunge's classification, after introducing a suitable notion of \emph{homotopy atoms}.

Diagram categories were used to lay the foundations of many important constructions in homotopy theory. For example, setting up stable homotopy theory begins with sequences of spaces, and the stable model structure on spectra is a localization of the projective model structure on diagrams \cite{BF}. The same is true equivariantly and motivically, where even the unstable homotopy theory requires diagram categories \cite{Hovey-spectra}. Diagram categories also arise in the homotopy theory of various categories of manifolds \cite{white-colimits}, and when studying morphisms in any model category \cite{white-yau5}. In monoidal settings, they arise when setting up the homotopy theory of operads \cite{white-yau2}, of algebras over colored operads \cite{white-yau1}, 
of (operad-structured) ideals of ring spectra \cite{white-yau6},
of polynomial monads \cite{quasi-tame, batanin-white-eilenberg},
and in higher category theory \cite{batanin-white-transactions, oberwolfach}. More generally, any combinatorial model category is Quillen equivalent to the Bousfield localization of a category of simplicial presheaves by a result by Dugger, \cite{Dugger-presentation}.

There are exceptions, however. An important example is the homotopy theory of pro-objects in a model category established by Edwards and Hastings, \cite{EH76}, and later on by Isaksen, \cite{Isaksen-strict}. In order to view it as (an opposite of) a subcategory of pro-representable functors in a category of diagrams, one has to consider the category of small presheaves indexed by a large category, as considered by the first author, \cite{Chorny-ClassHomFun,Chorny-ClassLinFun}. These categories of small functors indexed by large categories are neither locally presentable, nor cofibrantly generated, \cite{Chorny-maps}. They fall under a more general framework of class-combinatorial model categories developed by the first author and Rosicky, \cite{Chorny-Rosicky-II}.  In this work we deal only with diagram categories indexed by small categories. We hope to extend our theory to the categories of small functors in the future.

Hence, it is natural to want to characterize when a given model category is a diagram category, up to a Quillen equivalence.

In homotopy theory, the closest analog so far to Bunge's work was the concept of an orbit introduced by Dwyer and Kan, \cite{DK}. Their work produces a sufficient (but not necessary) condition for a category to have a model structure Quillen equivalent to the projective model structure on diagrams of spaces, indexed by a small category. By modifying the notion of an orbit we are able to characterize when a given model category is Quillen equivalent to the projective model structure on a presheaf category (see Theorem~\ref{criterion}). We work in the generality of $\cat V$-enriched model categories, rather than only simplicial model categories, and illustrate the reason for this generality with a number of examples. A different set of sufficient conditions for a model category to be Quillen equivalent to the presheaf category was worked out by Anna Montaruli, \cite{montaruli}.


As an application of Theorem~\ref{criterion} we  give a new and simple classification of polynomial functors from finite spaces to spectra as category of \emph{simplicial} presheaves over some indexing category. Our classification is not the first one, hence the need to compare it with the previous work.

There are several  approaches to the classification of the finitary polynomial functors from pointed spaces to spectra starting with an unpublished paper of Dwyer and Rezk, subsumed later on by the fundamental work of Arone and Ching, \cite{Arone-Ching-crosseffects}. In this paper, a careful analysis of the smash powers of the identity functor allowed the authors to relate the category of the $n$-truncated $\textup{Com}$-modules with the category of $\Gamma_{\leq n}$-indexed diagrams of spectra, where $\Gamma_{\leq n}$ is the category of pointed sets of size at most $n$. This is the work we have chosen to compare our results with. 

Our classification may be viewed as a variant of the above result with the exception that the diagrams we consider take values in a very convenient model of spectra built for this classification. 

Another paper by the same authors performs the classification in terms of an additional structure on the symmetric sequence of derivatives of a functor, \cite{Arone-Ching-classification}, and it is further away from our approach.

Other works on this topic include \cite{Arone-Barthel-classification} and \cite{Glasman} but these works concentrate on the classification of polynomial functors from spectra to spectra in terms of the Balmer spectrum and the category of Mackey functors, respectively.

We now describe the structure of the paper. In Section \ref{sec:prelim}, we recall the notion of orbits from \cite{DK}, and establish notation for the rest of the paper. In Section \ref{sec:atoms}, we prove our main theorem -- the classification of diagram categories up to homotopy -- and provide numerous examples connecting this result to equivariant homotopy theory, isovariant homotopy theory, stable homotopy theory, and $Cat$-enriched homotopy theory. In Section \ref{sec:classification}, we classify polynomial functors, our main application of Theorem \ref{criterion}. Lastly, in Section \ref{sec:comparison}, we compare our classification with other known classifications.

\section*{Acknowledgments}

The authors are grateful to John Harper and to the Ohio State University for hosting us in 2024, when this paper was conceived. We also thank Emmanuel Dror Farjoun and Bill Dwyer for helpful conversations. We thank Lennart Meier, Thomas Blom, Simon Henry, Omar Camarena, and Greg Arone for catching a mistake in an earlier version of this paper, and for directing our attention to related work.

\section{Preliminaries} \label{sec:prelim}

In this section, we introduce definitions and results that we will need to prove our main theorems, and we set down notation. We assume the reader is familiar with model categories (e.g., \cite{Hovey}), the basics of enriched model categories (e.g., \cite{Schwede-Shipley}), and the basics of combinatorial model categories (e.g., \cite{bous-loc-semi}).

Throughout the paper, $\sS$ will denote the category of simplicial sets, with the Kan-Quillen model structure. If $\cat V$ is a closed symmetric monoidal category, and $\cat M$ is a $\cat V$-model category \cite[Definition 4.2.1]{Hovey}, then for $X,Y \in \cat M$, we write $\hom(X,Y)$ for the hom object in $\cat V$. If $K\in \cat V$, we write $X\otimes K \in \cat M$ for the $V$-tensoring.

We next review the concept of an orbit, due to Dwyer and Kan \cite[2.1]{DK}. A set of orbits is defined to be a set $\{O_e\}_{e\in E}$ of objects of a simplicial category $\M$, that is closed under direct limits, has good homotopical behavior of transfinite compositions of pushouts of maps of the form $O_e \otimes K \to O_e \otimes L$ (where $K\to L$ is the inclusion of a subcomplex of a finite simplicial set), and $\hom(O_e,-)$ turns pushouts as above into homotopy pushout squares \cite[2.1]{DK}. A good example to keep in mind, which explains the name `orbit', is the set of objects $G/H$ where $G$ is a group and $H$ is a subgroup of $G$, so that $\hom(G/H,X)\simeq X^H$.

If a simplicial category $\M$ is equipped with a set of orbits, satisfying the axioms of \cite[2.1]{DK}, then there exists a model structure on $\M$ Quillen equivalent to the category of simplicial presheaves indexed by the full subcategory of orbits and equipped with the projective model structure \cite[Theorem 3.1]{DK}.

\begin{definition} \label{defn:orbit-model-str}
Let $\M$ be a simplicial category with a set $\{O_e\}_{e\in E}$ of orbits. The \textit{orbit model structure} on $\M$ \cite[Theorem 2.2]{DK} defines a morphism $f$ to be a weak equivalence (resp. fibration) if and only if the induced map $\hom(O_e,f)$ is a weak equivalence (resp. fibration) of simplicial sets. Cofibrations are characterized by the lifting property, and generating cofibrations are of the form $O_e \otimes i_n$ where $i_n:\partial \Delta[n] \to \Delta[n]$ is the usual inclusion.
\end{definition}

\begin{example} \label{ex:equivariant-orbits}
If $G$ is a topological group, then the set of spaces $\{G/H \;|\; H < G\}$ is a set of orbits. The standard model structure for $G$-spaces, modeling the equivariant homotopy theory in sense of Bredon \cite{Bredon}, is a special case of the orbit model structure, and Elmendorf's theorem \cite{elmendorf} is a special case of the Quillen equivalence \cite[Theorem 3.1]{DK}.
\end{example}

Another example of orbits is a collection of representable functors in the category of simplicial presheaves indexed by a small category \cat C. In this case the orbits define the projective model structure on the presheaf category $\sS^{\cat C^{op}}$. Generalizing the concept of homogeneous spaces to include actions of small categories,  Dror Farjoun and Zabrodsky came up with the following concept of an orbit (\cite[Definition 1.1]{DZ}): for a small category \cat D, a diagram $\dgrm T\in {\sS}$ is an orbit if $\colim_{\cat D}\dgrm T=\ast$. Dror Farjoun noticed later on that these new orbits are also orbits in the sense of Dwyer-Kan \cite{Farjoun}. The significant difference is that there is usually a proper class of orbits \cite{PhDI, Chorny-Dwyer}. The example of representable functors interpreted as orbits led to the development of the relative homotopy theory of Balmer and Matthey \cite{Balmer-Matthey} and was further extended to classes of orbits by the first author, in \cite{Chorny-relative}.

Together with the work of Gu \cite{gu2016generalized}, carrying over the concept of orbits to the categories enriched in $\Cat$, and work of Housden \cite{Housden}, using orbits for equivariant stable homotopy theory, these are all the currently known examples of orbits in the sense of Dwyer and Kan. It is evident, however, that the categories of diagrams of spaces (or spectra) appear frequently in homotopy theory and its applications and are not necessarily equipped with sets (or classes) of orbits. The axioms by Dwyer and Kan provide a sufficient condition allowing for a model structure Quillen equivalent to the category  of diagrams of spaces, but this condition is not necessary.

Having reviewed these preliminary concepts and results, we turn to our main theorem.

\section{Homotopy atoms} \label{sec:atoms}

In this section we define the notion of \textit{homotopy atoms} by softening the Dwyer and Kan orbit axioms \cite[2.1]{DK}, and we give a necessary and sufficient condition for a model category to be Quillen equivalent to a projective model structure on a category of presheaves taking values in a closed symmetric monoidal combinatorial model category. We let $\widehat{-}$ denote fibrant replacement in our ambient model category $\cat V$.

We say that a set of functors $\{F_i \,|\, i\in I\}$ \emph{jointly reflect} a property if, given a morphism $f$, if $F_if$ has the property for all $i\in I$, then so does $f$.

\begin{definition}\label{hom-atoms}
Let \cat V be a closed symmetric monoidal combinatorial model category with $\cat I = \{ A_i\cofib B_i \,|\, i\in I\}$ a set of generating cofibrations for some set $I$. Suppose that \cat M is a \cat V-model category. We say that \cat M is equipped with a set of \emph{homotopy atoms} if there exists a set of cofibrant objects $\calH\subset \cat M$ such that
\begin{enumerate}
\item the functors $\hom(T,-)$ for all $T\in \calH$ jointly reflect weak equivalences between fibrant objects;
\item the functors $\hom(T,\widehat{-})$ for all $T\in \calH$ commute, up to a weak equivalence, with $A_i\otimes -$ and $B_i\otimes - $ for all $i\in I$, with homotopy pushouts, and with sequential homotopy colimits.
\end{enumerate}
\end{definition}

\begin{remark}\label{rem}
We mostly have in mind (pointed) simplicial sets or spectra, $\cat V=\sS, \sS_\ast, \Sp$, as the base category, but the category of chain complexes or the category of small categories are also good examples. Note that in case  $\cat V=\sS$ or $\sS_\ast$, the verification of the commutation with $A_i\otimes -$ and $B_i\otimes - $ for all $i\in I$ follows by an inductive argument similar to \cite[Lemma~3.1]{Chorny-BrownRep} or, more generally \cite[Lemma~4.2]{Chorny-ClassLinFun}, provided that the homotopy pushouts are preserved by the functors $\hom(T,\widehat{-})$. If $\cat V = \Sp$ this verification is redundant by the generalization of \cite[Lemma~7.2]{Duality}, since the set of generating cofibrations in spectra has compact domains and codomains (these are the same generating cofibrations as in the projective model structure). Thus, Spanier-Whitehead duality implies that $A\wedge X \simeq \hom(DA, X)$ for every fibrant $X\in \Sp$. Substitute $X = \hom(T, \widehat -)$ to obtain
\begin{align*}
A\wedge \hom(T, \widehat -) \simeq \hom(DA,\hom(T, \widehat -)) \cong  \hom(DA\otimes T, \widehat -)\\
 \cong \hom(T,\hom(DA, \widehat -)) \simeq \hom(T,  \widehat{A\otimes -}),
\end{align*}
where the last weak equivalence is an application of Spanier-Whitehead duality again (the homotopy category of compact spectra is self dual).

For $\cat M = \Cat$ or $\textup{Ch}_R$ the full verification is required, since weighted homotopy colimits are different from ordinary homotopy colimits, see \cite{weighted}.
\end{remark}

With this definition in hand, we are ready to formulate our main result, which is a homotopical classification of diagram categories.

\begin{theorem}\label{criterion}
Let \cat M be a $\cat V\!$-model category for a combinatorial base category $\cat V\!$. Then there exists a small $\cat V\!$-category \cat C with a Quillen equivalence 

\[\xymatrix{R\colon\cat M \ar@/_/[r]^{\bot} & {\cat V^{\cat C^{\op}}:\! L} \ar@/_/[l]}\]

if and only if the category \cat M is equipped with a set of homotopy atoms, assuming that the category of $\cat V\!$-valued presheaves $\cat V^{\cat C^{\op}}$ is equipped with the projective model structure.
\end{theorem}

\begin{proof}
Suppose first that there is a Quillen equivalence of \cat M with the category of \cat V-valued presheaves. Then put $\calH=\{T_C=L(\hom(-,C) \,|\, C\in \cat C\}$. This is a set of cofibrant objects in \cat M, since the representable functors are cofibrant in the projective model structure. 

Note that the right adjoint $R\colon \cat M\to \cat V^{\cat C^{\op}}$ for every $M\in \cat M$ may be computed as $RM(C)=\hom(T_C, M)$, since, by Yoneda's lemma,
\[
\hom_{\cat M}(L(\hom(-,C)), M)=\hom_{\cat V^{\cat C^{\op}}}(\hom_{\cat C}(-,C), RM)=RM(C).
\] 
%
%

The right Quillen functor $R$ reflects weak equivalences between fibrant objects in \cat M and, for every cofibrant $F\in \cat V^{\cat C^{\op}}$, the derived unit of the adjunction $F\to R\widehat{LF}$ is a weak equivalence \cite[Corollary~1.3.16(c)]{Hovey}. Therefore, the functors $\hom(T,-)$, $T\in \calH$ jointly reflect weak equivalences between fibrant objects.

In order to show that the functor $\hom(T, \widehat{-})$ commutes with the homotopy colimits for all $T\in \calH$, it suffices to show that $R\widehat{-}$ commutes with homotopy colimits. For all $M\in \cat M$  there exists a zigzag of weak equivalences $M\weq \widehat{M}\leftweq L({R\widehat M})_{\mathrm{cof}}$ by \cite[Corollary~1.3.16(b)]{Hovey}. Therefore,
\[
R{(\hocolim_{i\in I} M_i)_{\mathrm{fib}}}\simeq 
R{(\hocolim_{i\in I}L(R\widehat{M_i})_{\mathrm{cof}})_{\mathrm{fib}}}\simeq
R{(L\hocolim_{i\in I}(R\widehat{M_i})_{\mathrm{cof}})_{\mathrm{fib}}}.
\]
The latter is weakly equivalent to $\hocolim_{i\in I}((R\widehat{M_i})_{\mathrm{cof}})\simeq \hocolim_{i\in I}(R\widehat{M_i})$ by \cite[Corollary~1.3.16(c)]{Hovey} again. In this argument, we denote by $\hocolim(-)$ the weighted homotopy colimit, i.e., $A\otimes -$ is a kind of homotopy colimit for a cofibrant $A\in\cat V$ (see \cite{weighted}).

For the inverse direction, assume that $\cat M$ is equipped with a set of homotopy atoms \calH. Let $\cat C$ be a full \cat V-subcategory of \cat M on the set of objects $\calH$, and consider the adjunction $\xymatrix{R\colon\cat M \ar@/_/[r]^{\bot} & {\cat V^{\cat C^{\op}}:\! L} \ar@/_/[l]}$, where $RM(T) = \hom(T,M)$ for all $T\in \calH$ and $M\in \cat M$. Hence, the left adjoint is given by $L(-)=(-)\otimes_{\cat C}H$, where $H\colon \cat C\to \cat M$ is the inclusion of the full subcategory \cite[3.5]{Kelly}.

This adjunction is a Quillen pair, since the right adjoint $R$ readily preserves fibrations and trivial fibrations, since those are levelwise in the projective model structure. In order to show that this is a Quillen equivalence we verify the conditions of \cite[Corollary~1.3.16(c)]{Hovey}. The right adjoint reflects weak equivalences by the first property of the homotopy atoms. It suffices to check that the map $F\to R\widehat{LF}$ is a weak equivalence for all cellular $F\in \cat V^{\cat C^{\op}}$, since any cofibrant object is a retract of a cellular one. By a straightforward cellular induction, this map a weak equivalence on each stage of the cellular construction, which is a combination of cotensors with the generating cofibrations of \cat V, of homotopy pushouts, and of sequential homotopy colimits. They are all preserved strictly by $L$ as a left Quillen functor and they are also preserved by $R(\widehat -)$, up to weak equivalence, by the second property of homotopy atoms.
\end{proof}

We now give several examples illustrating the power of Theorem \ref{criterion}. The first example is to the model categories that inspired the original definition of orbits.

\begin{example}
Dwyer-Kan orbits in a simplicial category \cat M are homotopy atoms with respect to the model structure they induce on \cat M, of Definition \ref{defn:orbit-model-str}.
\end{example}

Another famous classification in homotopical algebra is the Schwede-Shipley classification of stable model categories. Diagrams play an essential role here, because modules over a ring with many objects are best encoded as diagrams. Thus, the connection to Theorem \ref{criterion} should not be entirely surprising.

\begin{example} \label{ex:stable-classification}
Let \cat M be a stable simplicial model category equipped with a set of (weak) compact generators $\calG$ in the sense of Schwede and Shipley \cite{Schwede-Shipley}. Consider the Quillen equivalence $\xymatrix{F_0\colon\cat M \ar@/_/[r]^{\top} & {\Sp^{\Sigma}{\cat M}:\! \Ev_0} \ar@/_/[l]}$ of \cat M with a model category enriched over the category $\cat V = \Sp^{\Sigma}$ of symmetric spectra \cite[Def.~7.3]{Hovey-spectra}. Since the generators are defined on the level of the homotopy category, which did not change,  the set $F_0\calG$ forms a set of homotopy generators in the category $\Sp^{\Sigma}{\cat M}$. By \cite[Theorem~3.9.3]{Schwede-Shipley}, this spectral category is Quillen equivalent to the category of spectral presheaves indexed by the endomorphism category $\calE((F_0\calG)_{fib})$, which is a full subcategory of $\Sp^{\Sigma}{\cat M}$ generated by the fibrant replacements of the generators. Hence, the category $\Sp^{\Sigma}{\cat M}$ may be equipped with a set of homotopy atoms by Theorem~\ref{criterion}. 
\end{example}

A monoidal version of \cite[Theorem~3.9.3]{Schwede-Shipley} has recently been proven by \cite{Bayindir-Chorny} (for the monogenic setting) and by \cite{chorny-white-stable-monoidal} (for the general setting). We note that \cite{chorny-white-stable-monoidal} also has a $\cat V$-version of \cite[Theorem~3.9.3]{Schwede-Shipley}, allowing for a generalization of Example \ref{ex:stable-classification} away from the simplicial context. It would be interesting to formulate a monoidal version of Theorem \ref{criterion}, in analogy with \cite[Theorem A]{chorny-white-stable-monoidal}.

Our next example shows how to use Theorem~\ref{criterion} to provide a new 
model for spectra.
Even though this is not an example of a presheaf category, it is close enough and will be used in Section~\ref{sec:classification}.

In this case the conditions of Theorem~\ref{criterion} are not fully met, so we will use Bousfield localization to obtain a new model of spectra as a localization of the category of certain diagrams of spaces with respect to a smaller set of maps than provided by Dugger, \cite{Dugger-presentation}.

\begin{notation}
In the next example and further occurrences of stable model categories in this paper, we will use the desuspension notation for the derived version of the loop functor:
\[
\Sigma^{-i}(-) = \left( \Omega^i (-)_{\text{fib}}\right)_{\text{cof}}.
\]
\end{notation}

\begin{example}\label{spectra}
Let $\cat M = \Sp$, be the Bousfield-Friedlander model category of spectra enriched over $\cat V=\sS_\ast$, \cite{BF}. Assume that the underlying category for spectra is the category of presheaves over the simplicial category of spheres \Sph, whose objects are $\Ob(\Sph) = \NN$, and whose hom-objects are
\[
\hom_{\Sph}(i,j)= 
\begin{cases}
S^{j-i}, 	&  i\leq j \\
*,  		&  i>j.
\end{cases}
\]
Then every simplicial functor $X_\bullet\in \sS_\ast^{\Sph}$ is equipped with a natural map 
$$S^1=\hom_{\Sph}(i,i+1) \to \hom(X_i, X_{i+1}), \forall i\geq 0,$$
or, by adunction, with a map $\Sigma X_i = S^1\wedge X_i\to X_{i+1}$. 

Consider the set of objects $\calH = \{\Sigma^{-i}(\Sigma^\infty S^0) \,|\, i\geq 0 \}$. This set satisfies several of the properties of a set of homotopy atoms. We denote by \cat E the full subcategory of spectra on the set of objects \calH. There is a Quillen map similar to Theorem~\ref{criterion} 
\[
\xymatrix{
\sS_\ast^{\cat E^{\op}}
\ar@/^10pt/[r]^L
&
\Sp,
\ar@/^10pt/[l]^R
}
\]
where $R(X_\bullet)(i) = \hom(\Sigma^{-i}(\Sigma^\infty S^0), X_\bullet)$ for $i\geq 0$.

Let us show that the elements of $\calH$ jointly reflect weak equivalences of $\Omega$-spectra (the fibrant objects). Notice first that if $X_\bullet$ is an $\Omega$-spectrum, then $\Sigma^i X_\bullet \simeq (X_i,X_{i+1},\ldots)$, which is readily verified by application of $\Omega^i$ on both sides. 

Let $f_\bullet\colon X_\bullet \to Y_\bullet$ be a map of $\Omega$-spectra. If the maps 
$$\hom(\Sigma^{-i}(\Sigma^\infty S^0), f_\bullet)\simeq \hom(\Sigma^\infty S^0, \widehat{\Sigma^{i}(f_\bullet)}) \simeq  f_i$$ 
are weak equivalences for all $i\in \mathbb{N}$, then $f_\bullet$ is a projective weak equivalence of $\Omega$-spectra, hence a stable weak equivalence. Alternatively, notice that there is a weak equivalence $\Sigma^{-i}\Sigma^\infty S^0 \weq (\ast, \ldots, \ast, S^0, S^1,\ldots) = \hom_{\Sph}(i, - )$ to the representable functor. In other words there is a Dwyer-Kan equivalence between $\cat E^{op}$ and $\Sph$.

Moreover, homotopy pushouts of spectra are homotopy pullbacks. Thus, the functors $\hom(\Sigma^\infty S^i, \widehat -)$ take homotopy pushouts to homotopy pullbacks. There is no Quillen equivalence with the projective model structure, but we will show that it is possible to find a set $\calF$ of maps such that the (left Bousfield) localization with respect to $\calF$ turns the homotopy pullbacks into homotopy pushouts. Consider the class of homotopy pullback spans $\calP = \{A\fibr B \twoheadleftarrow C\,|\, A,B,C\in(\sS_\ast^{\cat E^{\op}})_{\text{fib}}\}$. Since $\calP$ is the collection of fibrant objects in the injective (or Reedy) model structure on the category $\sS_\ast^{\cdot\to \cdot\leftarrow \cdot}$, then the full subcategory on $\calP$ is $\lambda$-accessible as a small injectivity class for some cardinal $\lambda$. Let $\calP_\lambda\subset \calP$ be the subset of $\lambda$-presentable objects. For every $(A\fibr B \twoheadleftarrow C)\in \calP_\lambda$ put $D=A\times_B C$ and factor each of the pullback morphisms
$D\to A$ and $D\to C$ into a cofibration followed by a trivial fibration. The intermediate terms of these factorization we denote by $A'$ and $C'$, respectively. Finally, take $P = A'\coprod_D C'$ and consider the set $\calF = \{P\to B \,|\,(A\fibr B \twoheadleftarrow C)\in \calP_\lambda\}$. Then the left Bousfield localization with respect to $\calF$ ensures that the homotopy pullbacks become homotopy pushouts in $\sS^{\cat E^{\op}}$.

The Quillen adjunction $L\dashv R$ remains a Quillen adjunction after the localization, since the left Quillen functor $L$ preserves the cofibrations (that do not change) and the new trivial cofibrations. This is because $L$ preserves  homotopy pushouts and homotopy pullbacks in the projective model structures (before the localization on both sides) as a Quillen equivalence, implying that the maps in $\calF$ are sent by $L$ into the stable weak equivalences  in \Sp.

After this localization, the Quillen adjunction becomes a Quillen equivalence by the same argument as in the ``sufficient'' direction in the proof of Theorem~\ref{criterion}.
 \end{example}

Our next example is related to a new use of orbits, to isovariant homotopy theory, which we will describe.

\begin{example} \label{ex:isovar}
Yeakel's isovariant homotopy theory \cite{isovariant} and her isovariant Elmendorf's theorem can be viewed as a special case of Theorem \ref{criterion}, in analogy with Example \ref{ex:equivariant-orbits}. Let $G$ be a finite group. Let $\cat V$ be $\sS$, the category of simplicial sets. Let $\cat M$ be Yeakel's {\tt isvt-Top}, the category of compactly generated $G$-spaces with isovariant maps (and an added formal terminal object), i.e., equivariant maps $f:X\to Y$ with an equality of stabilizers $G_x = G_{f(x)}$ for all $x\in X$. Let $O$ be Yeakel's link orbit category $\mathcal{L}_G$ \cite[Definition 2.1]{isovariant}. This category contains all orbits $G/H$ but not all maps between them, because the goal is isovariant rather than equivariant homotopy theory. Yeakel's model structure on $\cat M$ \cite[Theorem 3.2]{isovariant} is a special case of the orbit model structure of Definition \ref{defn:orbit-model-str}, and the Quillen equivalence of her isovariant Elmendorf's theorem \cite[Theorem 4.1]{isovariant} is a special case of Theorem \ref{criterion}. 
\end{example}

\begin{remark}
There are many interesting questions that can be formulated for isovariant homotopy theory, with Example \ref{ex:isovar} in place. For instance, one could create spectra on isovariant spaces, following the model of \cite{hovey-white}, and obtain a stable version of the isovariant Elmendorf's theorem as a special case of Theorem \ref{criterion}. Or, one could work out a global homotopy theory, e.g., following \cite{global}. It is also possible to investigate the monoidal properties of the category of isovariant spaces, and work out the homotopy theory of operads and algebras in the isovariant context (which should also have versions of Theorem \ref{criterion}), following \cite{white-commutative, gutierrez-white, white-localization}. Another option would be to work out spectral Mackey functors in the isovariant context.
\end{remark}

Our next example is for model categories enriched in the category of small categories (or acyclic categories, or posets), and illustrates the value of working with $\cat V$-model categories in Theorem \ref{criterion} rather than only with simplicial model categories as we did in \cite{Chorny-White}.

\begin{example}
In 2016, Gu studied orbit model structures in the case where $\M$ is the category $Cat$ of small categories, or $Ac$ of acyclic categories, or $Pos$ of posets. Gu's main result \cite[Theorem 1.1]{gu2016generalized} proves the existence of the orbit model structure of Definition \ref{defn:orbit-model-str} on $Cat^I$, where $I$ is a small category and $O$ is a set of orbits (or, more generally, a locally small class), and further proves that the Thomason Quillen equivalence $\sS \leftrightarrows Cat$ induces a Quillen equivalence $\sS^I \leftrightarrows Cat^I$ with the orbit model structures. Since $Cat$ is a simplicial category (in more ways than one), the existence of Gu's model structure follows from \cite{Chorny-Dwyer}, since one can replace the internal hom of $Cat$ by its nerve, in Gu's proof. Furthermore, because \cite{gu2016generalized} was never published, we mention that the Quillen equivalence of \cite[Theorem 1.1]{gu2016generalized} follows from Theorem \ref{criterion} (to reduce the question to a Quillen equivalence of presheaf categories), together with \cite[Proposition 1.3]{Chorny-White} (in place of \cite[Lemma 4.2]{gu2016generalized}) and the argument of \cite[Proposition 1.4]{Chorny-White} (in place of \cite[Lemma 4.3]{gu2016generalized}) to compare the model structures on presheaf categories. The same holds for $Ac, Pos$ \cite[Theorem 6.1]{gu2016generalized}, and for the $G$-projective model structure on any of the three choices for $\M$ for a discrete group $G$ (\cite[Proposition 7.1]{gu2016generalized} for the model structure and \cite[Corollary 7.3]{gu2016generalized} for the Quillen equivalence), since these are again orbit model categories.
\end{example}

\begin{remark}
While we are on the topic of Gu's unpublished paper, we wish to point out that \cite[Remark 5.11]{gu2016generalized} is wrong, because the orbit model structure on $Cat^I$ will be a simplicial model structure (i.e., satisfy axiom SM7, by \cite[Theorem 2.2]{DK}), but that's not true for the Thomason model structure with Gu's structure. The issue is that the $(sd, Ex)$ enrichment (where $sd$ is for `subdivision') is not a simplicially enriched adjunction, since $sd$ does not preserve simplicially enriched colimits, since finite products in $\sS$ are tensors. To see that the SM7 axiom fails for the Thomason model structure, let $A$ be the one-point category, note that it's Thomason cofibrant, and note that the internal hom satisfies $Fun(A,B) = B$ for any $B$. Now, if $B$ were Thomason fibrant, and if the SM7 axiom held, it would imply that the simplicial mapping space $hom(A,B) = NFun(A,B) = NB$ is a Kan complex. But that only happens if $B$ is a groupoid, and not every Thomason fibrant $B$ must be a groupoid, e.g., any category with pushouts is Thomason fibrant. The same issue arises if one uses the simplicial mapping space $hom(A,B) = Ex^2 N Fun(A,B)$. The error is in the last line of \cite[Remark 5.11]{gu2016generalized}, where Gu states that the orbit model structure on $Cat^I$, where $I$ is the one-point category and $O = \{\ast\}$, coincides with the Thomason model structure. In fact, it coincides with the discrete model structure which, like the folk model structure, does satisfy the SM7 axiom. No choice of orbits can yield the Thomason model structure as an orbit model category.
\end{remark}

\section{A classification of (finitary) polynomial functors} \label{sec:classification}

In this section we apply Theorem~\ref{criterion} to another stable model category, the category of simplicial functors (enriched over the category of pointed simplicial sets, $\cat V = \sS_*$) from finite pointed simplicial sets to the category of symmetric spectra, $\Sp^{\sS_*^{\textup{fin}}}$. The goal is to classify the polynomial functors in the sense of Goodwillie, \cite{Goo:calc1}, as diagrams of spectra.

\begin{remark}\label{reduced}
We would like to stress that our functors are enriched over the category of \emph{pointed} simplicial sets $\sS_*$, hence, it follows that all functors $F\in \Sp^{\sS_*^{\textup{fin}}}$ are reduced, i.e., $F(\ast) = 0$. This follows from the representation of $F$ as a weighted colimit of representable functors: 
\[
F(X) = \int^{Y\in \sS_*^{\textup{fin}}} R^X(Y)\wedge F(X),\text{ where } R^X(Y) = \hom(X,Y).
\] 
\end{remark}

There are other classifications of polynomial functors from pointed spaces to spectra. Dwyer and Rezk  showed that polynomial functors are equivalent to the functors indexed by the category of finite sets and surjections (unpublished); a different set of classification results is due to Arone and Ching \cite{Arone-Ching-classification, Arone-Ching-crosseffects}. They show that the homotopy category of polynomial functors is equivalent to the category of symmetric sequences of spectra equipped with an additional structure of a divided power right module over the operad formed by the derivatives of the identity on based spaces, as in \cite{Arone-Ching-classification}, or, alternatively, they consider the category of coalgebras in symmetric sequences of spectra over the comonad $\mathrm{C_{KE_\bullet}}$, where $\mathrm{KE_n}$ is the Koszul dual of the little $n$-disc operad and $\mathrm{KE_\bullet}$ is an inverse sequence of operads, or a pro-operad, as in \cite{Arone-Ching-crosseffects}. For $\cat V$-enriched contexts, the homotopy theory of algebras over operads is described in \cite{white-yau1, white-yau3}, and for coalgebras over cooperads, with connections to Koszul duality, in \cite{white-yau7}.

\begin{lemma}\label{derivative}
Let
$\calH = \left\{\left. \Sigma^{-p}\left(\Sigma^\infty (\bigwedge_{i=1}^k R^{S^0})_{\textup{cof}}\right) \,\right |\, 1\leq k\leq n,\, p\geq 0 \right\} \subset {\Sp^{\sS_*^{\textup{fin}}}}$ be a set of objects, and  let $f\colon F\to G$ be a map of fibrant functors in ${\Sp^{\sS_*^{\textup{fin}}}}$ equipped with the $n$-excisive model structure, \cite[Theorem 4.6]{BCR}. If the induced map $$\hom(H, f)\colon \hom(H, F)\to \hom(H, G)$$ is a weak equivalence of (pointed) simplicial sets for all $H\in \calH$, then the map $f$ is a projective weak equivalence. In other words, the objects of the set $\calH$ jointly reflect weak equivalences of fibrant objects.
\end{lemma}
\begin{proof}
Put $H_{k,p}= \Sigma^{-p}\left(\Sigma^\infty (\bigwedge_{i=1}^k R^{S^0})_{\textup{cof}}\right)$.
Recall from \cite[Lemma~8.2(ii)]{BCR} that for any projectively fibrant $F\in\Sp^{\sS_*^{\textup{fin}}}$, the following holds:
\begin{align*}
\hom_{\Sp^{\sS_*^{\textup{fin}}}}\left( H_{k,p},F\right) &= \hom_{\Sp^{\sS_*^{\textup{fin}}}}\left( \Sigma^{-p}\Sigma^\infty(\bigwedge_{i=1}^k R^{S^0})_{\textup{cof}}, F\right)\\ 
&\simeq \hom_{\sS_*^{\sS_*^{\textup{fin}}}}\left( (\bigwedge_{i=1}^k R^{S^0})_{\textup{cof}}, \Omega^\infty  \widehat{\Sigma^p F}\right) \\ 
&= cr_k (\Omega^\infty \widehat{\Sigma^p F})(S^0,\ldots S^0) \\ 
&\simeq \Omega^\infty \left( \Sigma^p(cr_k F(S^0,\ldots, S^0))\right)_{\textup{fib}}.
\end{align*}

Similarly to Example~\ref{spectra}, $\Omega^\infty \left( \Sigma^p(cr_k F(S^0,\ldots, S^0))\right)_{\textup{fib}}$ is equivalent to the $p$-th layer of the spectrum $cr_k F(S^0,\ldots, S^0)_{\textup{fib}}$. Let $f\colon F\to G$ be a natural transformation of projectively fibrant functors from finite spaces to spectra such that $\hom(H_{k,p}, f)$ is a weak equivalence for all $1\leq k\leq n$, $p\geq 0$.  Fix $k$ and $q$ and let $p$ run from $0$ to $\infty$. Then we obtain a weak equivalence of spectra 
\begin{equation}\label{cr-k-f}
cr_k f(S^0,\ldots S^0)\colon cr_k F(S^0,\ldots S^0)\to cr_k G(S^0,\ldots S^0)
\end{equation}
for all $1\leq k\leq n$, as these spectra are levelwise weakly equivalent after a fibrant replacement. 

Since the model category of polynomial functors $\Sp^{\sS_*^{\text{fin}}}$ is stable (the suspension commutes with polynomial approximation), the fibre sequence $D_kF\to P_kF\to P_{k-1}F$ is part of the exact triangle $D_kF\to P_kF\to P_{k-1}F\to \Sigma^{-1}D_kF$. We will use this to prove that $f: F\to G$ is a weak equivalence. Put $\Sigma^{-1}D_kF = R_kF$, cf \cite[Lemma~2.2]{Goo:calc3}, and consider the fibre sequence $P_kF\to P_{k-1}F\to R_kF$, where $R_kF$ is a $k$-homogeneous functor. The same construction applies to $G$, and $f\colon F\to G$ induces a morphism of exact triangles. Assume for the sake of induction that $P_{k-1}f\colon P_{k-1}F\to P_{k-1}G$ is a weak equivalence. The base case is satisfied since both functors $F$ and $G$ are reduced. 

Notice that the weak equivalence (\ref{cr-k-f}) implies, in particular, that 
\[
cr_k R_kf\colon cr_k R_kF(S^0,\ldots, S^0)\to cr_k R_kG(S^0,\ldots,S^0)
\]
is a weak equivalence, hence, by \cite[Proposition~5.8]{Goo:calc3}, the induced map of the multilinear functors
\[
cr_k R_kf\colon cr_k R_kF(X_1,\ldots, X_k)\to cr_k R_kG(X_1,\ldots,X_k)
\]
is a weak equivalence for all $X_1\ldots, X_k\in \sS_*^{\text{fin}}$. Therefore, \cite[Proposition~3.4]{Goo:calc3} implies that the induced map of the $k$-homogeneous functors $R_kf\colon R_kF\to R_kG$ is a weak equivalence. Hence, the induced map of the homotopy fibers is a weak equivalence. 

We conclude by induction that, for any $n$, the map $f\colon F\to G$ of $n$-excisive functors is a weak equivalence.

\end{proof}

\begin{proposition}\label{Kelly-split}
The full subcategory \cat C of the category of simplicial functors $\Sp^{\sS_\ast^{\text{\emph{fin}}}}$ on the set $\calH$ of objects defined in Lemma~\ref{derivative}, may be decomposed, up to a Dwyer-Kan equivalence, into a Kelly product (\cite[6.5]{Kelly}) of two categories: the category \cat E from Example~\ref{spectra} and the full subcategory $\cat F\subset \Sp^{\sS_\ast^{\text{\emph{fin}}}}$ on the set of objects $\left\{\left.  \Sigma^\infty (\bigwedge_{i=1}^k R^{S^0})_{\textup{cof}} \right|\, 1\leq k\leq n \right\}$.

\end{proposition}
\begin{recall}
Let \cat V be a closed symmetric monoidal category. The Kelly product of two $\cat V$-categories \cat A and \cat B is a \cat V-category $\cat A\otimes \cat B$ with $\obj{\cat A\otimes \cat B}=\obj{\cat A}\times \obj{\cat B}$ and $\hom_{\cat A\otimes \cat B}((A, B), (A', B')) = \hom_{\cat A}(A, A')\otimes \hom_{\cat B}(B, B')$, \cite[Section~1.4]{Kelly}. Moreover, if  \cat A, \cat B,  \cat C are \cat V-categories, then the exponential rule is satisfied, \cite[Section~6.5]{Kelly}.
\[
\cat C^{\cat A\otimes \cat B}\cong \left(\cat C^{\cat A}\right)^{\cat B}.
\]
\end{recall}
\begin{proof}[Proof of \ref{Kelly-split}]
Let us put
\[
C_{p,k} = \Sigma^{-p} \left( \Sigma^\infty (\bigwedge_{i=1}^{k} R^{S^{0}})_{\textup{cof}}\right)\in \cat C, E_p = \Sigma^{-p} (\Sigma^\infty S^0), \text{ and } F_{k} =  \Sigma^\infty (\bigwedge_{i=1}^k R^{S^0})_{\textup{cof}} \in \cat F,
\]
for $1\leq k\leq n$ and $p \geq 0$.

Now we are going to establish a Dwyer-Kan equivalence of simplicial categories $T\colon \cat E \wedge  \cat F \to  \cat C$, assigning $T(E_p,  F_{k}) = C_{p,k}$ and for each pair of objects 
\begin{align*}
T_{(E_{p_1},  F_{k_1}), (E_{p_2},  F_{k_2})}\colon \hom_{\cat E\wedge \cat F}\left(
(E_{p_1},  F_{k_1}),
(E_{p_2},  F_{k_2})
\right)
 \longrightarrow  
\hom_{\cat C}(
C_{p_1,k_1}, C_{p_2,k_2})
\end{align*}
is assigned to be a natural weak equivalence of pointed simplicial sets, since
\begin{align*}
\hom_{\cat E\wedge \cat F}\left(
(E_{p_1},  F_{k_1}),
(E_{p_2},  F_{k_2})
\right) = \\
 \hom_{Sp}(\Sigma^{-p_1} (\Sigma^\infty S^0),\Sigma^{-p_2} (\Sigma^\infty S^0)) \wedge
 \hom_{\Sp^{\sS_*^{\text{fin}}}}(F_{k_1}, F_{k_2})
 \\
 = \begin{cases}
  S^{p_1-p_2}\wedge
 \hom_{\Sp^{\sS_*^{\text{fin}}}}(F_{k_1}, F_{k_2}), &  \text{if } p_1 \geq p_2; \\
  \ast,  & \text{if } p_1<p_2
\end{cases}
\end{align*}
and $\hom_{\cat C}(C_{p_1,k_1}, C_{p_2,k_2}) = \hom_{\cat C}((\bigwedge_{i=1}^{k_1} R^{S^{0}})_{\textup{cof}},\Omega^{\infty}\Sigma^{p_1-p_2}(\Sigma^{\infty}(\bigwedge_{i=1}^{k_2} R^{S^{0}})_{\textup{cof}}))$. By \cite[Lemma~8.2(ii)]{BCR}, this is $cr_{k_1}(\Omega^{\infty}\Sigma^{p_1-p_2}(\Sigma^{\infty}(\bigwedge_{i=1}^{k_2} R^{S^{0}})_{\textup{cof}}))(S^{0},\ldots, S^{0})$. In case $p_1< p_2$, this cross-effect is contractible, since any suspension spectrum is connective and its desuspensions can only produce contractible entries at the $0$-th level, therefore the cross-effect of a contractible diagram is contractible. In case $p_1\geq p_2$, we notice that $(p_1-p_2)$-fold suspension in spectra may be modeled as a smash product with a simplicial sphere, which may be viewed as a homotopy colimit, i.e., it commutes with cross-effects for functors taking values in spectra, since it may be equivalently expressed as a co-cross-effect (a finite sequence of homotopy colimits), \cite[Definition~1.5]{Arone-Ching-crosseffects}. Therefore, if $p_1\geq p_2$, then 
\begin{align*}\hom_{\cat C}(C_{p_1,k_1}, C_{p_2,k_2}) \simeq \Omega^{\infty}(S^{p_1-p_2}\wedge cr_{k_1}(\Sigma^{\infty}(\bigwedge_{i=1}^{k_2} R^{S^{0}})_{\textup{cof}})(S^{0},\ldots, S^{0})) \\
= S^{p_1-p_2}\wedge\Omega^{\infty}cr_{k_1}(\Sigma^{\infty}(\bigwedge_{i=1}^{k_2} R^{S^{0}})_{\textup{cof}})(S^{0},\ldots, S^{0})\\
 = S^{p_1-p_2}\wedge \hom_{\Sp^{\sS_*^{\text{fin}}}}(F_{k_1}, F_{k_2}).
\end{align*}

In other words, there is a natural weak equivalence of simplicial sets 
\[
T_{C_{p_1,k_1},C_{p_2,k_2}}\colon \hom_{\cat E\wedge \cat F}\left(
(E_{p_1},  F_{k_1}),
(E_{p_2},  F_{k_2})
\right)
 \weq 
 \hom_{\cat C}(
C_{p_1,k_1}, C_{p_2,k_2})
\]
for every pair of objects of \cat C, or $T$ is a Dwyer-Kan equivalence.
\end{proof}

\begin{theorem}\label{thm:classification-polynomial}
The category of functors $\Sp^{\sS_*^{\text{\emph{fin}}}}$ equipped with the $n$-excisive model structure is Quillen equivalent to the projective model structure on the category $\Sp^{\cat F^{\op}}$.
\end{theorem}

\begin{proof}
We will construct a zigzag of Quillen equivalences. Consider the subcategory $\cat C$ of $\Sp^{\sS_\ast^{\text{fin}}}$ on the set of objects $\calH$ defined in Lemma~\ref{derivative}.  By Proposition~\ref{Kelly-split}, there is a Dwyer-Kan equivalence of categories $T\colon \cat E\wedge \cat F \to \cat C$ or of the opposite categories $T^{\op}\colon \cat E^{\op}\wedge \cat F^{\op} \to \cat C^{\op}$. This defines a Quillen equivalence $\Lan_{T^{op}} \dashv (T^{op})^\ast$ of the presheaf categories with the projective model structure by \cite[Theorem~2.1]{Dwyer-Kan-Equiv}. Similarly to Theorem~\ref{criterion}, we define a Quillen adjunction $L\dashv R$ where $R: \Sp^{\sS_{\ast}^{\text{fin}}} \to \sS_\ast^{\cat C^{\op}}$ is defined by $R(F)(C) = \hom(C,F)$ for all $C\in \cat C$. The left adjoint $L$ is given by $L(-)=(-)\otimes_{\cat C}H$, where $H\colon \cat C\to \cat \Sp^{\sS_\ast^{\text{fin}}}$ is the inclusion. We thus have the following chain of Quillen maps:

\[
\xymatrix{
\left(\sS_{\ast}^{\cat E^{\op}}\right)^{\cat F^{op}} 
\ar@/^10pt/[r]^>>>{\Lan_{T^{op}}}
&
\sS_\ast^{\cat C^{\op}}
\ar@/^10pt/[l]^<<<{(T^{op})^\ast}
\ar@/^10pt/[r]^>>>L
&
\Sp^{\sS_{\ast}^{\text{fin}}}.
\ar@/^10pt/[l]^<<<R
}
\]

Recall that the localization with respect to the set $\calF$ of maps from Example~\ref{spectra} has ensured that the category $\sS_*^{\cat E^{\op}}$ becomes Quillen equivalent to spectra. The category $\left(\sS_{\ast}^{\cat E^{\op}}\right)^{\cat F^{op}}$ may be localized with respect to the set $\calF_1 = \{\hom(-,F)\wedge f \,|\, F\in \cat F,\, f\in \calF\}$ so that it becomes Quillen equivalent to the category of spectrum valued functors $\Sp^{\cat F^{op}}$, but the Quillen map from Example~\ref{spectra} goes in the wrong direction. Instead we consider the composition of the adjunctions $L_1 = L\Lan_{T^{\op}}\dashv (T^{\op})^*R = R_1$, which remains a Quillen map after the localization of $\left(\sS_{\ast}^{\cat E^{\op}}\right)^{\cat F^{op}}$ with respect to $\calF_1$, since the cofibrations do not change and trivial cofibrations are preserved by the left Quillen functor $L_1$ similarly to Example~\ref{spectra}. We thus have the following Quillen pair:

\begin{equation}\label{Q-adj-first}
\xymatrix{
\Sp^{\cat F^{\op}}
\ar@/^10pt/[r]^{L_1}
&
\Sp^{\sS_{\ast}^{\text{fin}}}
\ar@/^10pt/[l]^{R_1}
},
\end{equation}

Let $f\colon F\to G$ be a map of fibrant $n$-excisive functors. Assume that the natural transformation of diagrams of fibrant spectra $(T^{\op})^*(R(f))$ is a weak equivalence. Then $R(f)$ is a levelwise weak equivalence by \cite[Corollary 1.3.16]{Hovey}. By Lemma~\ref{derivative} the right adjoint $R$ reflects weak equivalences, hence $f$ is weak equivalence and the right adjoint $R_1$ of the Quillen adjunction (\ref{Q-adj-first}) reflects weak equivalences of fibrant objects.

In order to show that (\ref{Q-adj-first}) is a Quillen equivalence, we have to show that the derived unit of the adjunction $X \to R_1\widehat{L_1 X}$ is a weak equivalence for every cofibrant $X\in \Sp^{\sS_\ast^{\text{fin}}}$ by \cite[Corollary 1.3.16]{Hovey}. The proof proceeds analogously to the proof of the corresponding part in Theorem~\ref{criterion}.

Let us show that $R_1$ preserves homotopy pushouts. It obviously preserves homotopy pullbacks of fibrant objects. Since homotopy pushouts are also homotopy pullbacks in the stable model category for $n$-excisive functors, and the target model category is also stable, it follows that homotopy pullbacks in diagrams of spectra are homotopy pushouts. By Remark~\ref{rem}, $R_1(\widehat{-})$ also commutes, up to a weak equivalence, with $\partial \Delta_+^k\wedge - $ and $\Delta_+^k\wedge - $. Finally $R_1(\widehat{-})$ preserves filtered colimits, up to a weak equivalence. In other words, the functors $\Sigma^\infty(\bigwedge_{i=1}^k R^{S^0})$ behave pretty much like homotopy atoms, except that the functors they produce take values in spectra instead of simplicial sets. Hence, if we start from a cellular object $X\in \Sp^{\cat F^{\op}}$ and proceed by a cellular induction, the functor $L_1$ preserves the cellular structure as a left Quillen functor and the functor $R_1(\widehat{-})$ preserve each stage of the cellular construction, up to a weak equivalence, by the above commutation properties. Hence, $X \to R_1\widehat{L_1 X}$ is a weak equivalence for every cofibrant $X\in \Sp^{\sS_\ast^{\text{fin}}}$ and the theorem follows from \cite[Corollary~1.3.16(c)]{Hovey}.
\end{proof}

\begin{remark}
Note that, by the nature of the Kelly product \cite[6.5]{Kelly}, the right adjoint $R_1$ may be thought of as assigning spectrum-valued cross-effects to each $F\in \Sp^{\sS_\ast^{\text{fin}}}$, cf. \cite[Lemma~3.13]{Arone-Ching-crosseffects}, \cite[Lemma~8.3(ii)]{BCR}. In the notation of \cite[Lemma~3.13]{Arone-Ching-crosseffects} we have, up to Quillen equivalence:
\begin{equation}\label{Q-adj}
\xymatrix{
\Sp^{\cat F^{\op}}
\ar@/^10pt/[r]^{L_1}
&
\Sp^{\sS_{\ast}^{\text{fin}}}
\ar@/^10pt/[l]^{R_1 = \{\Nat(\Sigma^\infty(\wedge_{i=1}^k R^{S^0})_{\text{cof}}, - )\}_{k=1}^n}
},
\end{equation}
where $\Nat(-,-)$ is the spectrum of natural transformations. In the notation of \cite[Lemma~8.3(ii)]{BCR}, 
$$\Nat\left(\Sigma^\infty(\wedge_{i=1}^k R^{S^0})_{\text{cof}}, - \right) = \spt\left((\wedge_{i=1}^k R^{S^0})_{\text{cof}}, - \right),$$
where for $K\in \sS_\ast^{\sS_\ast^{\text{fin}}}$ and $F\in \Sp^{\sS_\ast^{\text{fin}}}$, the spectrum $\spt(K,F)$ is defined by:
\[
\spt(K,F)_i = \hom(K, \Ev_i \circ F) = \hom(K, \hom(\mathbb{S}^{-i}, F)), \, \mathbb{S}^{-i} = \Sigma^{-i}(\Sigma^\infty S^0).
\]
Note that we can not conclude that the right adjoint preserves fibrations of fibrant objects and trivial fibrations, because these mapping spaces do not define a spectral enrichment.
\end{remark}

\section{Comparison to other classifications of polynomial functors} \label{sec:comparison}

This section is devoted to the comparison of our classification of polynomial functors to other results in this field.

Now we are able to compare our classification result to the Dwyer-Rezk classification of polynomial functors using the results of Arone and Ching, \cite{Arone-Ching-crosseffects}.

Let $\Omega_{\leq n}$ denote the category of non-empty finite sets with at most $n$ points and surjections as morphisms. And let  $\Omega_{\leq n}^+$ denote the category with the same objects as $\Omega_{\leq n}$, and morphisms $\hom_{\Omega_{\leq n}^+}(m,k) = \hom_{\Omega_{\leq n}}(m,k)_+\in \sS_\ast$. The reason for adding the base point is to have an enrichment over $\sS_\ast$ on $\Omega_{\leq n}^+$.

Recall that the categories $\Sp$ and $\sS_*$ may be enriched over the category of simplicial sets $\sS$, as well as over the category of pointed simplicial sets, using the half-smash product instead of the smash product.

\begin{corollary}
The $\sS_*$-category of functors $\Sp^{\sS_*^{\text{\emph{fin}}}}$ equipped with the $n$-excisive model structure is Quillen equivalent to the projective model structure on the $\sS$-category $\Sp^{\Omega_{\leq n}}$.
\end{corollary}
\begin{proof}
Let $\mathbb{S}$ denote the sphere spectrum. Consider the weak equivalence  
\[
\phi\colon \bigvee_{m\fibr k } \mathbb{S} \overset {\sim} {\longrightarrow} \textup{Nat}(F_k, F_m),\, m,k \leq n
\]
which appeared in \cite[3.16]{Arone-Ching-crosseffects}. Looking at the $0$-th level of $\phi$ we obtain the weak equivalence of simplicial sets
\[
\textup{Ev}_0\circ\phi \colon \hom_{\Omega_{\leq n}^+}(m,k)\longrightarrow \hom_{\Sp^{\sS_*^{\text{fin}}}}(F_k, F_m)
\]
In other words, there is a Dwyer-Kan equivalence of the $\sS_*$-categories $\Omega_{\leq n}^+$ and $\cat F^{\op}$ inducing a Quillen map between the categories of spectra valued functors with the projective model structure. That implies a Dwyer-Kan equivalence between the subcategories of fibrant-cofibrant objects of $\Sp^{\Omega_{\leq n}^+}$ and $\Sp^{\cat F^{\op}}$ with the projective model structure. The latter is Quillen equivalent to the $n$-excisive model structure on the category of $\sS_*$-functors $\Sp^{\sS_*^{\text{fin}}}$ by Theorem~\ref{thm:classification-polynomial}.

The last reduction to the original Dwyer-Rezk classification is that the underlying category of the $\sS_*$-category $\Sp^{\Omega_{\leq n}^+}$ is naturally equivalent to the $\sS$-category $\Sp^{\Omega_{\leq n}}$ and the projective model structures on both underlying categories coincide.
\end{proof}

\bibliographystyle{abbrv}
\bibliography{Xbib}

\end{document}

%% file: definitions.tex
\newcommand{\Sp}{\ensuremath{\textup{Sp}}}
\newcommand{\spt}{\ensuremath{\textup{spt}}}
\newcommand{\Nat}{\ensuremath{\textup{Nat}}}
\newcommand{\Sph}{\ensuremath{\textup{Sph}}}
\newcommand{\Ob}{\ensuremath{\textup{Ob}}}

\newcommand{\op}{{\ensuremath{\textup{op}}}}

\newcommand{\dgrm}[1]{\ensuremath{\smash{\underset{\widetilde{\hphantom{#1}}}{#1}} \mathstrut}}

\newcommand {\cofib} {\ensuremath{\hookrightarrow}}

\newcommand {\fibr} {\ensuremath{\twoheadrightarrow}}

\newcommand {\weq} {\ensuremath{\tilde\rightarrow}}
\newcommand {\leftweq} {\ensuremath{\tilde\leftarrow}}

\DeclareMathOperator{\End}{\textup{End}}

\newtheorem {theorem1}{Theorem}[section]
\newtheorem {theorem}[theorem1]{Theorem}
\newtheorem {corollary}[theorem1]{Corollary}
\newtheorem {proposition}[theorem1]{Proposition}
\newtheorem {lemma}[theorem1]{Lemma}

\theoremstyle{definition}
\newtheorem {definition}[theorem1]{Definition}
\newtheorem {example}[theorem1]{Example}

\theoremstyle{remark}

\newtheorem {remark}[theorem1]{Remark}
\newtheorem {recall}[theorem1]{Reminder}
\newtheorem {notation}[theorem1]{Notation}

\newcommand{\calC}{\ensuremath{\mathcal{C}}}

\newcommand{\calE}{\ensuremath{\mathcal{E}}}
\newcommand{\calF}{\ensuremath{\mathcal{F}}}
\newcommand{\calG}{\ensuremath{\mathcal{G}}}
\newcommand{\calH}{\ensuremath{\mathcal{H}}}

\newcommand{\calP}{\ensuremath{\mathcal{P}}}

\newcommand{\M}{\ensuremath{\mathcal{M}}}

\newcommand{\cat}[1]{\ensuremath{{\mathscr #1}}}

\newcommand{\sS}{\ensuremath{\mathcal{S}}}

\newcommand{\Cat}{\ensuremath{\calC\textup{at}}}

\DeclareMathOperator{\Ev}{\ensuremath{\textup{Ev}}}

\newcommand{\NN}{\ensuremath{\mathbb{N}}}

\newcommand{\colim}{\ensuremath{\mathop{\textup{colim}}}}
\newcommand{\hocolim}{\ensuremath{\mathop{\textup{hocolim}}}}

\DeclareMathOperator{\Lan}{\ensuremath{\textup{Lan}}}

\newcommand{\obj}[1]{\ensuremath{\textup{obj}(#1)}}

\renewcommand{\hom}{\ensuremath{{\rm hom}}}

\newcounter{zahl}%
    {\end{list}}%